\documentclass[a4paper]{amsart}

\usepackage{amsmath,amssymb,amscd}
\usepackage{mathtools}
\usepackage{amsthm}
\theoremstyle{plain}
\newtheorem{theorem}{Theorem}[section]
\newtheorem{proposition}[theorem]{Proposition}
\newtheorem{corollary}[theorem]{Corollary}
\newtheorem{lemma}[theorem]{Lemma}
\theoremstyle{definition}
\newtheorem{definition}[theorem]{Definition}
\newtheorem{example}[theorem]{Example}
\newtheorem{remark}[theorem]{Remark}

\numberwithin{equation}{section}

\DeclareMathOperator{\image}{Im}
\DeclareMathOperator{\kernel}{Ker}
\DeclareMathOperator{\trace}{tr}
\DeclareMathOperator{\pfaffian}{pf}

\newcommand{\eigensp}[2]{V_{#1}(#2)}
\newcommand{\geigensp}[2]{\tilde{V}_{#1}(#2)}
\newcommand{\sympleigensp}[4]{V_{{#2},{#2}^{*#1}}(#3,#4)}

\newcommand{\symplgeigensp}[4]{\tilde{V}_{{#2},{#2}^{*#1}}(#3,#4)}

\newcommand{\symplcharpoly}{\varphi}

\newcommand{\cfield}{\mathbb{K}}
\newcommand{\ccfield}{\mathbb{K}(s)}
\newcommand{\field}{\mathbb{F}}
\newcommand{\alter}{A}
\newcommand{\zeromatrix}[1]{O}
\newcommand{\idmap}[1]{E}
\newcommand{\sendo}[2]{({#1}-s\idmap{})({#2}-s\idmap{})}
\newcommand{\ssendo}{({M}-s\idmap{})({M}^{*\omega}-s\idmap{})}
\newcommand{\scharpoly}[3]{\varphi_{\sendo{#1}{#2}}({#3})}
\newcommand{\exspace}{V\otimes_{\cfield}\ccfield}
\usepackage[pdftex]{hyperref}
\hypersetup{%
 setpagesize=false,%
 bookmarks=true,%
 bookmarksdepth=tocdepth,%
 bookmarksnumbered=true,%
 colorlinks=true,
 hidelinks, 
 pdftitle={},%
 pdfsubject={},%
 pdfauthor={},%
 pdfkeywords={}}

\makeatletter
\@namedef{subjclassname@2020}{
	\textup{2020} Mathematics Subject Classification}
\makeatother

\begin{document}

\title{The symplectic characteristic polynomial}
\author{Kohei Ichizuka}
\address{Department of Mathematics, Saitama University, 255 Shimo-Okubo, Sakura-Ku,
Saitama 338-8570, Japan}
\email{k.ichizuka.929@ms.saitama-u.ac.jp}



\date{\today}


\begin{abstract}
We introduce the notion of the symplectic characteristic polynomial of an endomorphism of a symplectic vector space. 
This is a polynomial in two variables and can be considered as a generalization of the characteristic polynomial of the endomorphism in the context of symplectic linear algebra. One of the goal of this paper is to  prove that the symplectic characteristic polynomial is a complete symplectic invariant of symplectically diagonalizable endomorphisms.
\end{abstract}

\maketitle

\section{Introduction}
Let $(V,\omega)$ be a finite dimensional symplectic vector space over a field $\cfield$.
It is a fundamental question to classify endomorphisms of $V$ with respect to symplectic similarity. 
Endomorphisms which are self-adjoint, anti-self-adjoint or preserve the symplectic form $\omega$ have been classified 
when the characteristic of $\cfield$ is not $2$
(see 
\cite[Theorem~4]{waterhouse2005structure}, 
\cite[Theorem~2.2]{ciampi1973classification},
\cite[Section~9]{springer2022symplectic}).
Such an endomorphism is symplectically normal (Definition~\ref{definition sympl similar and normal} (2)) and the classification of symplectically normal endomorphisms is not known even though they are diagonalizable.
An endomorphism is symplectically diagonalizable if and only if
the endomorprhism is symplectically normal and diagonalizable (\cite[Theorem~13]{de2016diagonalizability}).
A motivation of the paper is to classify symplectically  diagonalizable endomorphisms up to symplectic similarity. 
\par
Let us recall several facts on the characteristic polynomial of an endomorphism $M$ of a vector space $V$.
Cayley-Hamilton theorem implies that $V$ admits  the direct sum  decomposition into the generalized eigenspaces of $M$ when all eigenvalues of $M$ are in $\cfield$.
The characteristic polynomial is a complete invariant of diagonalizable endomorphisms. Moreover, if the number of eigenvalues of $M$ is equal to the dimension of $V$, then $M$ is diagonalizable.
\par
Now we consider a symplectic vector space $(V,\omega)$. Let $M$ be an endomorphism of $V$.
Suppose that $M$ is symplectically normal and that all eigenvalues of $M$ are in $\cfield$. It can be proved by using the same idea in (\cite[Lemma~1]{gohberg1990classification}) that $V$ is the symplectically orthogonal sum of the symplectic subspaces associated with two eigenvalues of $M$. This implies that we have projections onto these subspaces. 
On the other hand, if one considers the endomorphism $(M-s\idmap{})(M^{*\omega}-s\idmap{})$ of the $\ccfield$-vector space  $V\otimes_{\cfield}\ccfield$, we have other projections, whose images are  the generalized eigenspaces of $\ssendo$.
Here $E$ is the identity map and $\ccfield$ is the field  of  rational functions over $\cfield$.
\par
In this paper,
we show that these generalized eigenspaces and projections correspond to the previous symplectic subspaces and projections respectively.
%
We introduce the symplectic characteristic polynomial as 
the Pfaffian of  $(M-s\idmap{})^*\omega-t\omega$.
We prove that 
the symplectic characteristic polynomial is a complete  invariant of symplectically diagonalizable endomorphisms up to symplectic similarity.
Moreover, we give a sufficient condition for symplectically diagonalizability.
%
\par
The paper is organized as follows. Section~\ref{symplectic linear algebra} contains  
several fundamental facts on symplectic linear algebra. 
In Section~\ref{symplectic characteristic polynomial}, we define the symplectic characteristic polynomial of an endomorphism and prove several properties. 
%
In Section~\ref{characteristic polynomial of two endomorphisms}, we study the characteristic 
polynomial of 
the endomorphism 
associated with
two endomorphisms $M$, $N$ of a (not necessary symplectic) vector space.
In Section~\ref{Applications of the symplectic characteristic polynomial}, 
we 
prove that the symplectic characteristic polynomial is a complete invariant of symplectically diagonalizable endomorphisms
and give a sufficient condition on symplectically diagonalizability.
\par

\section{Preliminaries}\label{symplectic linear algebra}
We first introduce the following notation. 
\begin{itemize}
\item
We denote  the field  of  rational functions over $\cfield$ by $\ccfield$.
\item
The set of  endomorphisms of a vector space $V$ is denoted by $\operatorname{End}(V)$. By abuse of notation, $\idmap{}$ denotes the identity map of $V$ and $O$ denotes the zero map of $V$.
\item
The characteristic polynomial of an endomorphism $M$ is denoted by $\varphi_M(t)$. 
\item
Let $M,N\in \operatorname{End}(V)$ and $\lambda, \mu \in \cfield$. We define the subspaces $\eigensp{M}{\lambda}$, $\geigensp{M}{\lambda}$ by
\begin{align*}
\eigensp{M}{\lambda}&=\kernel(M-\lambda E),\\
\geigensp{M}{\lambda}&=\kernel(M-\lambda E)^n \ \text{where $\dim_{\cfield} V=n$}.
\end{align*}
The space $\eigensp{M}{\lambda}$ (resp. $\geigensp{M}{\lambda}$) is an eigenspace (resp. a generalized  eigenspace ) if it is not zero.
\par
We also define the subspaces $\eigensp{M,N}{\lambda,\mu}$, $\geigensp{M,N}{\lambda, \mu}$ of $V$ by
\begin{align*}
\eigensp{M,N}{\lambda,\mu}&=\left(\eigensp{M}{\lambda} \cap \eigensp{N}{\mu}\right)+\left(\eigensp{M}{\mu} \cap \eigensp{N}{\lambda}\right), \\
\geigensp{M,N}{\lambda, \mu}&=\left(\geigensp{M}{\lambda} \cap \geigensp{N}{\mu}\right) +\left( \geigensp{M}{\mu} \cap \geigensp{N}{\lambda}\right).
\end{align*}
\item Let $M\in \operatorname{End}(V)$ and  $\omega$  a bilinear form.  Let $\field$ be a extension field of $\cfield$.  We use the same letter $M$ (resp. $\omega$) for the natural extension 
of $M$ (resp. $\omega$)
to the endomorphism (resp. the bilinear form) of $V\otimes_{\cfield}\field$.
\end{itemize}
%
%
%
%
%
%
%
%
%
\par
Now we recall several fundamental facts on symplectic linear algebra. 
\par
Let $V$ be a vector space over a field $\cfield$ and let $\dim_{\cfield} V=2n$. A bilinear form $\omega$ of $V$ is said to be  \textit{symplectic} if 
\begin{itemize}
\item 
$\omega$ is alternating, that is, $\omega(v,v)=0$ for all $v\in V$ and
\item 
$\omega$ is  non-degenerate, that is, a linear map 
\[
\omega^{\flat}: V\rightarrow V^{*} \quad v  \mapsto \omega(\cdot,v)
\]
is an isomorphism.
\end{itemize} 
 We call $(V,\omega)$ a \textit{symplectic vector space}. 
A representation matrix of  an alternating matrix with respect to some bases is anti-symmetric and the diagonal entries are zero.  
This matrix has determinant zero when the size of the matrix is odd and a symplectic vector space must be even-dimensional. 
\par
The set of automorphisms which preserve $\omega$ is 
denoted by
$\operatorname{Sp}(V,\omega)$. Since the Pfaffian of a matrix which represents $\omega$ is not zero, any automorphism in $\operatorname{Sp}(V,\omega)$ has determinant $1$.
\par
Let $(V,\omega)$ be a symplectic vector space and let $M\in \operatorname{End}(V)$. The \textit{symplectic 
adjoint endomorphism} $M^{*\omega}$ of $M$ is an endomorphism  of $V$ defined by
\[
M^{*\omega}=(\omega^{\flat})^{-1}M^*\omega^{\flat}
\]
where $M^*$ is the dual map of $M$. 
It is easily shown that the following properties hold for any $M, N \in \operatorname{End}(V)$.
\begin{itemize} 
\item
$\omega(M v,w)=\omega(v,M^{*\omega}w)$ for all $v,w \in V$. 
\item 
$M\in \operatorname{Sp}(V,\omega)$ if and only if $M^{*\omega}=M^{-1}$.
\item  
$(M^{*\omega})^{*\omega}=M.$
\item  
$(MN)^{*\omega}=N^{*\omega}M^{*\omega}$.
\item 
The characteristic polynomial of $M$ is equal to that of $M^{*\omega}$.
\end{itemize} 
\begin{definition}\label{definition sympl similar and normal}
\quad
\begin{itemize}
\item [(1)]
An endomorphism $M$ is said to be \textit{symplectically similar} to an endomorphism $N$ if there exists  $P\in \operatorname{Sp}(V,\omega)$ such that $P^{-1}MP=N$.
\item [(2)]
An endomorphism $M$ is said to be \textit{symplectically normal} if $MM^{*\omega}=M^{*\omega}M$.
\end{itemize}
\end{definition}
\par 
Let $W$ be a subspace of $V$. We define the \textit{symplectically 
orthogonal subspace} $W^{\perp\omega}$ 
by 
\[
W^{\perp\omega}=\{v\in V | \ \omega(v,w)=0 \quad \forall w \in W\}.
\]
As in the case of an inner product, we have the following. 
\begin{proposition}\label{propperp}
Let 
$M\in \operatorname{End}(V)$
 and let $W$  be a subspace of $V$. Then
\begin{itemize}
\item [$(1)$] $\dim W+\dim W^{\perp\omega}=\dim V$.
\item [$(2)$] $(W^{\perp\omega})^{\perp\omega}=W$.
\item [$(3)$] $\kernel M=(\image M^{*\omega})^{\perp\omega}$.
\item [$(4)$] $M(W) \subset W$ implies that $M^{*\omega}(W^{\perp\omega}) \subset W^{\perp\omega}.$
\item [$(5)$]  If $\lambda\neq \mu$, then 
$
\geigensp{M}{\lambda}\subset \geigensp{M^{*\omega}}{\mu}^{\perp\omega}.
$
\end{itemize}
\end{proposition}
\begin{proof}
$(1)$
It follows that $W^{\perp\omega}=\kernel \iota_W^*\omega^{\flat}$ where $\iota_W$ is the inclusion.
Since a linear map $\iota_W^*\omega^{\flat}:V\rightarrow W^*$ is surjective,
we have $(1)$.
$(2)$
It is clear that $W \subset (W^{\perp\omega})^{\perp\omega}$. Applying $(1)$ for $W^{\perp\omega}$, we have $\dim (W^{\perp\omega})^{\perp\omega}=\dim W$. Hence, we get $(2)$.
$(3)$
It is easy to see that $\kernel M\subset(\image M^{*\omega})^{\perp\omega}$.
Let $v_0\in (\image M^{*\omega})^{\perp\omega}$. Then $\omega(M v_0, v)=\omega(v_0, M^{*\omega}v)=0$ for all $v\in V$.
This implies that $v_0 \in \kernel M$. Hence, we have $(3)$.
The items
$(4)$ are easily shown.
$(5)$
Suppose that $\lambda \neq \mu$.
Since $(M-\mu \idmap{})^{2n}|_{\geigensp{M}{\lambda}}$ is an automorphism, 
$\geigensp{M}{\lambda}=\image(M-\mu \idmap{})^{2n}\cap \geigensp{M}{\lambda}\subset\image(M-\mu \idmap{})^{2n}$.
Therefore, by $(2)$ and $(3)$, we get $(5)$.
\end{proof}

A subspace $W$ of $V$ is said to be 
\begin{itemize}
\item 
\textit{symplectic} if $W\cap W^{\perp\omega}=\{0\}$, 
\item
\textit{isotropic} if $W\subset W^{\perp\omega}$, 
\item
\textit{coisotropic} if $W^{\perp\omega}\subset W$,
\item
\textit{Lagrangian} if $W= W^{\perp\omega}$. 
\end{itemize}
\par
Since $(W^{\perp\omega})^{\perp\omega}=W$, we have that 
$W$ is symplectic if and only if $W^{\perp\omega}$ is symplectic.\par
The item $(1)$ in 
Proposition~\ref{propperp}  shows that if $W$ is isotropic, then $\dim_{\cfield}W\leq n$ and that  $W$ is Lagrangian if and only if $W$ is isotropic and $\dim_{\cfield} W=n$.
\begin{proposition}\label{prop skew ham sympl subsp}
Let $S\in \operatorname{End}(V)$. Suppose that $S^{*\omega}=S$.
Then $\geigensp{S}{\lambda}$ is a symplectic subspace of $(V,\omega)$.
\end{proposition}
\begin{proof}
It follows from the item $(2)$ and $(3)$ in
Proposition~\ref{propperp} that $\geigensp{S}{\lambda}^{\perp\omega}=\image(S-\lambda E)^{2n}$.
This implies that $\geigensp{S}{\lambda}\cap \geigensp{S}{\lambda}^{\perp\omega}=\{0\}$.
\end{proof}

\begin{lemma}\label{lemma Lag split}
Let $M\in \operatorname{End}(V)$ and let $W=\symplgeigensp{\omega}{M}{\lambda}{\mu}
$.
Suppose that
$\lambda\neq \mu$ 
and that $W$ is a  symplectic subspace of $(V,\omega)$.
Then 
$\geigensp{M}{\lambda}\cap \geigensp{M^{*\omega}}{\mu}$, $\geigensp{M}{\mu}\cap \geigensp{M^{*\omega}}{\lambda}$ are Lagrangian subspaces of $(V,\omega)$.
\end{lemma}
\begin{proof}
The item $(5)$ in Proposition~\ref{propperp} gives that 
$\geigensp{M}{\lambda}\cap \geigensp{M^{*\omega}}{\mu}$, $\geigensp{M}{\lambda}\cap \geigensp{M^{*\omega}}{\mu}$ are isotropic subspaces of 
$(W,\omega|_W)$. 
Since 
$\geigensp{M}{\lambda}\cap 
\geigensp{M}{\mu}
=\{0\}$, 
we have
\[
\dim_{\cfield}(\geigensp{M}{\lambda}\cap \geigensp{M^{*\omega}}{\mu})=\dim_{\cfield}(\geigensp{M}{\mu}\cap \geigensp{M^{*\omega}}{\lambda})=\frac{1}{2}\dim_{\cfield}W.
\]
Therefore, 
$\geigensp{M}{\lambda}\cap \geigensp{M^{*\omega}}{\mu}$, $\geigensp{M}{\mu}\cap \geigensp{M^{*\omega}}{\lambda}$ are Lagrangian subspaces of $(W,\omega|_W)$.
\end{proof}
A basis $\mathcal{B}=(e_1,\dots,e_n,f_1,\dots,f_n)$ of $V$ is said to be \textit{symplectic} if  
\[
\omega(e_i,e_j)=0, \ \omega(f_i,f_j)=0, \ \omega(e_i,f_j)=\delta_{ij}, \quad i,j\in\{1,\dots,n\}.
\]
\begin{proposition}\label{prop exist sympl basis}
There exists a symplectic basis of $(V,\omega)$.
\end{proposition}
\begin{proof}
We proof by induction on $n$.
Since $\omega$ is non-degenerate, there exist $e_1,f_1\in V$ such that $\omega(e_1,f_1)=1$. Hence, the assertion holds for $n=1$. We assume the assertion holds for $n-1$.
Let $W$ be a subspace spanned by $e_1, f_1$. Then $W$ is symplectic and so is $W^{\perp\omega}$. The induction hypothesis shows that  there exists a symplectic basis $(e_2,\dots, e_n, f_2,\dots, f_n)$ of $(W^{\perp\omega},\omega|_{W^{\perp\omega}})$.  It is clear that $(e_1,\dots, e_n, f_1,\dots, f_n)$ is a symplectic basis of $(V,\omega)$.
\end{proof}
\begin{lemma}\label{lemma sympl basis}
Let $L_1,L_2$ be Lagrangian subspaces of $(V,\omega)$. Suppose that  $L_1\cap L_2 =\{0\}$ and $(e_1,\dots, e_n)$ is a basis of  $L_1$. Then there exists a basis  
$
(f_1,\dots, f_n)$ of $L_2
$ such that  $(e_1,\dots,e_n, f_1, \dots, f_n)$ is a 
symplectic basis of $(V,\omega)$.
\end{lemma}
\begin{proof}
The condition $L_1\cap L_2=\{0\}$ implies that  there exist $\lambda_1,\dots, \lambda_n\in V^*$ such that
$
\lambda_i(e_j)=\delta_{ij}$ and that  $\ \lambda_i|_{L_2}=0.
$
Since $\omega$ is non-degenerate, there exists $f_i\in V$ such that $\lambda_i=\omega^{\flat}(f_i)$. It follows from $\lambda_i|_{L_2}=0$ that $f_i\in (L_2)^{\perp\omega}=L_2$. It is clear that $(e_1,\dots, e_n, f_1,\dots, f_n)$ is a symplectic basis.
\end{proof}
Let $\mathcal{B}$ be a symplectic basis of $(V,\omega)$. Then
the matrix $[\omega]_{\mathcal{B}}$ which represents $\omega$ with respect to $\mathcal{B}$ is
\[
[\omega]_{\mathcal{B}}=
\left(
\begin{array}{cc}
\zeromatrix{n} & E_n \\
-E_n & \zeromatrix{n}
\end{array}
\right)
\]
where $E_n$ is  the identity matrix of order $n$ and $\zeromatrix{n}$ is the zero matrix.
If  $M\in \operatorname{End}(V)$ is represented by
\[
[M]_{\mathcal{B}}=
\left(
\begin{array}{cc}
A_n & B_n \\
C_n & D_n 
\end{array}
\right)
\]
where $A_n$, $B_n$, $C_n$, $D_n$ are matrices of order $n$, then $M^{*\omega}$ is represented by
\[
[M^{*\omega}]_{\mathcal{B}}=[\omega]_{\mathcal{B}}^{-1}[M]_{\mathcal{B}}^T[\omega]_{\mathcal{B}}=
\left(
\begin{array}{cc}
D_n^T & -B_n^T \\
-C_n^T & A_n^T 
\end{array}
\right).
\]
Here $T$ denotes the transpose of a matrix. 
Hence, 
if $\dim_{\cfield}V=2$, then $MM^{*\omega}=(\det M)\idmap{}$, $M+M^{*\omega}=(\trace M)\idmap{}$.
\par
Let  $P\in \operatorname{End}(V)$ and let $\mathcal{B}=(e_1,\dots, e_n,f_1,\dots, f_n)$ be a symplectic basis of $(V,\omega)$. Then $(Pe_1,\dots, Pe_n, Pf_1,\dots, Pf_n)$ is a symplectic basis if and only if $P\in \operatorname{Sp}(V,\omega)$. This means that
\begin{proposition}\label{prop sympl similar}
An endomorphism $M$ is symplectically similar to an endomorphism $N$ if and only if there exist symplectic bases $\mathcal{B}$, $\mathcal{B'}$ such that $[M]_{\mathcal{B}}=[N]_{\mathcal{B'}}$.
\end{proposition}
%
%
\par
An endomorphism $M$  is said to be \textit{symplectically diagonalizable} if there exists a symplectic basis consisting of eigenvectors of $M$.
It is easy to see that if an endomorphism $M$ is symplectically diagonalizable, then $MM^{*\omega}=M^{*\omega}M$ and $M$ is diagonalizable.
Theorem~\ref{theorem complete inv} shows that 
the converse is true.
%
%
%
\par
Let $M\in \operatorname{End}(V).$ 
We define the bilinear form $\omega_M$ of $V$ by
\[
\omega_M(v,w)=\omega(v,Mw) \quad v,w \in V.
\] 
It is easy to see that $\omega_{MM^{*\omega}}=(M^{*\omega})^*\omega$.  
We remark that $\omega_{MM^{*\omega}}$, $\omega_{M+M^{*\omega}}$ are alternating for any $M\in \operatorname{End}(V)$.
\par
Let 
$\alter \in \operatorname{End}(V)$ and assume that  
$\omega_{\alter}$ is alternating.  Let $(e_1,\dots, e_n, f_1,\dots, f_n)$ be a symplectic basis. Set $(v_1,\dots, v_{2n})=(e_1,f_1,\dots, e_n,f_n)$. Define the matrix $\Omega_A$ by $\Omega_A=(\omega_A(v_i, v_j))_{i,j}$.
Since $\det P=1$ for $P\in \operatorname{Sp}(V,\omega)$, the Pfaffian $\pfaffian(\Omega_A)$ does not depend on the choice of symplectic bases.  We denote $\pfaffian(\Omega_A)$ by $\pfaffian_{\omega}(\omega_A)$. It is easily shown that $\pfaffian_{\omega}(\omega)=1$.
\begin{definition}\label{def root poly}
Let 
$\alter\in \operatorname{End}(V)$ and assume that 
$\omega_{\alter}$ is alternating. The polynomial $\psi_{\alter}(t)$ is defined by 
$\psi_{\alter}(t)=\pfaffian_{\omega}(\omega_{\alter-t\idmap{}})$,
\end{definition}
\begin{proposition}\label{skew poly}
Let 
$\alter\in \operatorname{End}(V)$ and assume that  
$\omega_{\alter}$ is  alternating.
 The polynomial $\psi_{\alter}(t)$ has the following properties.
\begin{itemize}
\item [$(1)$] 
The square $\psi_{\alter}(t)^2$
is the characteristic polynomial $\varphi_{\alter}(t)$. 
\item [$(2)$] 
The polynomial $\psi_{\alter}(t)$ is an invariant of $\alter$ under symplectic similarity transformation, that is, 
$\psi_{P^{-1}\alter P}(t)=\psi_{\alter}(t)$  for $P\in \operatorname{Sp}(V,\omega)$.
\item [$(3)$] 
Let $\psi_{\alter}(t)=a_0+\dots+a_{n-1}t^{n-1}+a_n t^n$. Then 
\[
a_0=\pfaffian_{\omega}(\omega_{\alter}), \ 2a_{n-1}=(-1)^{n-1}\trace {\alter}, \ a_n=(-1)^n.
\]
\item [$(4)$] 
$\psi_{\alter}(\alter)=O.$
\end{itemize}
\end{proposition}
\begin{proof}
It is clear from the definition of $\psi_{\alter}(t)$ that $(1)$ holds.
\par
Let $P\in \operatorname{Sp}(V,\omega)$.
Since $\omega_{P^{-1}AP}=P^*\omega$, we get that $\omega_{P^{-1}AP}$ is alternating.  The item $(1)$ implies that $\psi_{P^{-1}\alter P}(t)^2=\psi_{\alter}(t)^2$. Hence, we have $(2)$.
\par
It follows from \cite[Theorem 3]{waterhouse2005structure} that there exists a symplectic basis $\mathcal{B}$ 
such that 
\[
[\alter]_{\mathcal{B}}=
\left(
\begin{array}{cccc}
D_n^T & \zeromatrix{n} \\
\zeromatrix{n} & D_n
\end{array}
\right)
\]
where $D_n$ is a matrix of order $n$. 
Then
\[
[\omega_{\alter-t\idmap{}}]_{\mathcal{B}}
=[\omega]_{\mathcal{B}}[\alter-t \idmap{}]_{\mathcal{B}} 
=
\left(
\begin{array}{cccc}
\zeromatrix{n} & D_n-tE_n\\
-(D_n-tE_n)^T & \zeromatrix{n}
\end{array}
\right).
\]
This implies that $\pfaffian([\omega_{\alter-t\idmap{}}]_{\mathcal{B}})=(-1)^\frac{n(n-1)}{2}\det(D_n-tE_n)$.
Hence, we have
$
\psi_{\alter}(t)=\det(D_n-tE_n).
$
Since $\det D_n=\psi_{\alter}(0)=\pfaffian_{\omega}(\omega_{\alter})$ and $2\trace D_n=\trace \alter$, we get $(3)$.
Cayley-Hamilton theorem gives that $\psi_{\alter}(D_n)=O$.
Therefore, 
\begin{align*}
[\psi_{\alter}(\alter)]_{\mathcal{B}}
& =
\left(
\begin{array}{cc}
\psi_{\alter}(D_n)^T & \zeromatrix{n} \\
\zeromatrix{n} & \psi_{\alter}(D_n)
\end{array}
\right)
=\zeromatrix{2n},
\end{align*}
which completes the proof.
\end{proof}

\begin{proposition}\label{prop poly of prod}
Let $M\in \operatorname{End}(V)$. Then  we have $\psi_{MM^{*\omega}}(t)=\psi_{M^{*\omega}M}(t)=\pfaffian_{\omega}(M^*\omega-t\omega)$ and  $\psi_{MM^{*\omega}}(0)=\det M$.
\end{proposition}
\begin{proof}
By the item $(2)$ in
Proposition~\ref{skew poly}, we get $\psi_{MM^{*\omega}}(t)^2=\psi_{M^{*\omega}M}(t)^2$.  The item $(4)$ in Proposition~\ref{skew poly} shows that the coefficient of $t^n$ in $\psi_{MM^{*\omega}}(t)$ and the coefficient of $t^n$ in $\psi_{M^{*\omega}M}(t)$ are $(-1)^n$. Hence, $\psi_{MM^{*\omega}}(t)=\psi_{M^{*\omega}M}(t)$.
Since $\omega_{M^{*\omega}M}=M^*\omega$, we have $\psi_{M^{*\omega}M}(t)=\pfaffian_{\omega}(M^*\omega-t\omega)$. This implies that
$\psi_{MM^{*\omega}}(0)=\det M\cdot\pfaffian_{\omega}(\omega)=\det M$.
\end{proof}
%
%
%
%
%
\section{The symplectic characteristic polynomial}\label{symplectic characteristic polynomial}
Let $(V,\omega)$ be a $2n$-dimensional  symplectic vector space over a field $\cfield$.  
\begin{definition}\label{def sympl char}
The \textit{symplectic 
 characteristic polynomial} $\symplcharpoly_M^{\omega}(s,t)$ 
of an endomorphism $M$ of $(V,\omega)$ is defined by
\[
\symplcharpoly^{\omega}_M(s,t)=\psi_{\ssendo}(t).
\]
Here  $\psi_{\ssendo}(t)$ is a polynomial defined in Definition~\ref{def root poly}.
\end{definition}
\begin{remark}
It follows from Proposition~\ref{prop poly of prod}  that
\begin{equation}\label{eq sympl poly}
\psi_{\ssendo}(t)=\psi_{(M^{*\omega}-s\idmap{})(M-s\idmap{})}(t)=\pfaffian_{\omega}(M-s\idmap{})^*\omega-t\omega).
\end{equation}
\end{remark}
\begin{remark}
The definition above is motivated by the following example.
Let $\dim_{\cfield} V=4$ and $M\in \operatorname{End}(V)$. Assume that $M$ has four distinct eigenvalues  and that $M$ is symplectically diagonalizable, that is, there exists a symplectic basis $\mathcal{B}$ such that 
$[M]_{\mathcal{B}}$ is a diagonal matrix $\operatorname{diag}(\lambda_1,\lambda_2,\lambda_3,\lambda_4)$.
Define the subgroup $G$ of the symmetric group $\mathfrak{S}_4$ as follows:
An element
$\sigma\in\mathfrak{S}_4$ is in  $G$ if there exists a symplectic basis $
\mathcal{B}'$ such that $[M]_{\mathcal{B}'}=\operatorname{diag}(\lambda_{\sigma(1)},\lambda_{\sigma(2)},\lambda_{\sigma(3)},\lambda_{\sigma(4)})$.
Since $\trace MM^{*\omega}=\lambda_1\lambda_3+\lambda_2\lambda_4$ 
is independent on the choice of symplectic bases, $G$ consists of
eight elements 
\begin{align*}
\left(
\begin{array}{cccc}
1 & 2 & 3 & 4 \\
1 & 2 & 3 & 4
\end{array}
\right), \ 
\left(
\begin{array}{cccc}
1 & 2 & 3 & 4 \\
2 & 3 & 4 & 1
\end{array}
\right), \
\left(
\begin{array}{cccc}
1 & 2 & 3 & 4 \\
3 & 4 & 1 & 2
\end{array}
\right), \  
\left(
\begin{array}{cccc}
1 & 2 & 3 & 4 \\
4 & 1 & 2 & 3
\end{array}
\right), \\
\left(
\begin{array}{cccc}
1 & 2 & 3 & 4 \\
1 & 4 & 3 & 2
\end{array}
\right), \ 
\left(
\begin{array}{cccc}
1 & 2 & 3 & 4 \\
2 & 1 & 4 & 3
\end{array}
\right), \ 
\left(
\begin{array}{cccc}
1 & 2 & 3 & 4 \\
3 & 2 & 1 & 4
\end{array}
\right), \
\left(
\begin{array}{cccc}
1 & 2 & 3 & 4 \\
4 & 3 & 2 & 1
\end{array}
\right),  
\end{align*}
which implies that 
$G$ is isomorphic to the dihedral group of square. 
Therefore, we find that
$
\{\{\lambda_1, \lambda_3\},\{\lambda_2, \lambda_4\}\}
$
is independent on the choice of symplectic bases. 
This set determines uniquely the polynomial 
\[
p(s,t)=
\{(\lambda_1-s)(\lambda_3-s)-t\}\{(\lambda_2-s)(\lambda_4-s)-t\}.
\]
%
On the other hand, 
%
\begin{align*}
&[\ssendo]_{\mathcal{B}}\\
=& \operatorname{diag}((\lambda_1-s)(\lambda_3-s),(\lambda_2-s)(\lambda_4-s),(\lambda_1-s)(\lambda_3-s),(\lambda_2-s)(\lambda_4-s)).
\end{align*}
This implies that
$
\varphi_{\ssendo}(t)=p(s,t)^2.
$
The item $(4)$ in Proposition~\ref{skew poly} shows that the coefficient of $t^2$ in $\psi_{\ssendo}(t)$ is $1$. Hence, we have
$
\psi_{\ssendo}(t)=p(s,t).
$
\end{remark}
We show properties of the symplectic characteristic polynomial.
\begin{proposition}\label{prop symplectic char poly}
Let $M\in\operatorname{End}(V)$. The symplectic characteristic polynomial $\symplcharpoly_M^{\omega}(s,t)$ satisfies the following.
\begin{itemize}
\item [$(1)$] 
The polynomial $\symplcharpoly^{\omega}_M(0,t)$ is the polynomial $\psi_{MM^{*\omega}}(t)$.
\item [$(2)$]
The polynomial $\symplcharpoly_M^{\omega}(s,0)$ is the characteristic polynomial $\varphi_M(s)$.
\item [$(3)$] 
$
\symplcharpoly_M^{\omega}(s,(\sigma-s)(\tau-s))
=\psi_{MM^{*\omega}}(\sigma\tau)+\dots +\psi_{M+M^{*\omega}}(\sigma+\tau)(-s)^n.
$ 
\item [$(4)$] 
If $\symplcharpoly^{\omega}_M(s,(\lambda-s)(\mu-s))=0$ for $\lambda,\mu\in\cfield$, then $\lambda$, $\mu$ are eigenvalues of $M$, $\lambda\mu$ is an eigenvalue of $MM^{*\omega}$ and $\lambda+\mu$ is an eigenvalue of $M+M^{*\omega}$.
\item [$(5)$]
The symplectic characteristic polynomial is an invariant of $M$ under symplectic similarity transformation and taking the symplectic adjoint, that is, $\symplcharpoly^{\omega}_{P^{-1}MP}(s,t)=\symplcharpoly^{\omega}_M(s,t)$  for $P\in \operatorname{Sp}(V,\omega)$ and 
 $\symplcharpoly^{\omega}_{M^{*\omega}}(s,t)=\symplcharpoly^{\omega}_{M}(s,t)$.
\item [$(6)$] 
The square $\symplcharpoly^{\omega}_M(s,t)^2$
is the characteristic polynomial $\varphi_{\ssendo}(t)$.
\item [$(7)$]
Let $\symplcharpoly^{\omega}_M(s,t)=a_0(s)+\dots +a_{n-1}(s)t^{n-1}+a_n(s)t^n.$ Then
\begin{align*}
a_0(s)&=\varphi_M(s), \ 2a_{n-1}(s)=(-1)^{n-1}\{\trace MM^{*\omega}-2(\trace M)s+n s^2\},\\ 
\ a_n(s)&=(-1)^n.
\end{align*}
\item [$(8)$]
$\symplcharpoly^{\omega}_M(s,\ssendo)=O.$
\end{itemize}
\end{proposition}
\begin{proof}
It is clear from the definition that $(1)$ holds. 
The item $(2)$ follows from Proposition~\ref{prop poly of prod}.
\par
$(3)$ 
The easy calculation shows that 
\[
(M^{*\omega}-s \idmap{})^*\omega-(\sigma-s)(\tau-s)\omega=\omega_{MM^{*\omega}-\sigma\tau E}-s\omega_{M+M^{*\omega}-(\sigma+\tau)E}.
\] 
This implies that 
\begin{align*}
&\symplcharpoly_M^{\omega}(s,(\sigma-s)(\tau-s))
=
\pfaffian_{\omega}
\left(
\omega_{MM^{*\omega}-\sigma\tau  E}-s\omega_{\{M+M^{*\omega}-(\sigma+\tau)E\}}
\right) \\
=&
\psi_{MM^{*\omega}}(\sigma\tau)+\dots +\psi_{M+M^{*\omega}}(\sigma+\tau)(-s)^n.
\end{align*}
\par
$(4)$
Assume that $\symplcharpoly^{\omega}_M(s,(\lambda-s)(\mu-s))=0$ for $\lambda,\mu\in\cfield$. Then by $(2)$, we have  $\varphi_M(\lambda)=\symplcharpoly^{\omega}_M(\lambda,0)=0$.
From $(3)$, we get $\psi_{MM^{*\omega}}(\lambda\mu)=0$, $\psi_{M+M^{*\omega}}(\lambda+\mu)=0$.
This implies that $\varphi_{MM^{*\omega}}(\lambda\mu)=0$, $\varphi_{M+M^{*\omega}}(\lambda+\mu)=0$.
\par
$(5)$
Let $P\in \operatorname{Sp}(V,\omega)$.
It is easy to see that
\[
(P^{-1}MP-s \idmap{})(P^{-1}M^{*\omega}P-s \idmap{})=P^{-1}(M-s\idmap{})(M^{*\omega}-s\idmap{})P.
\] 
Hence,  from the item $(1)$ in Proposition~\ref{skew poly}, we have $\symplcharpoly^{\omega}_{P^{-1}MP}(s,t)=\symplcharpoly^{\omega}_M(s,t)$.
The equality \eqref{eq sympl poly}  
implies that $\symplcharpoly^{\omega}_{M^{*\omega}}(s,t)=\symplcharpoly^{\omega}_{M}(s,t)$.
\par
The items
$(6)$--$(8)$ follows from the items $(2)$--$(4)$ in Proposition~\ref{skew poly} respectively. 
\end{proof}
\par
The symplectic characteristic polynomial $\symplcharpoly^{\omega}_M(s,t)$ of an endomorphism $M$ is determined by the characteristic polynomial $\varphi_M(t)$
if $M^{*\omega}=M$, $M^{*\omega}=-M$ or $M^{*\omega}=M^{-1}$. Precisely, we have the following proposition.
\begin{proposition}\label{sympl char special case}
Let $M\in \operatorname{End}(V)$.
Then
\[
\symplcharpoly^{\omega}_M(s,t)^2=
\begin{dcases}
\varphi_M(s-t^{\frac{1}{2}})\varphi_M(s+t^{\frac{1}{2}}) & \text{if $M^{*\omega}=M$}, \\
\varphi_{M^2}(s^2-t)
 & \text{if $M^{*\omega}=-M$}, \\
s^{2n}\varphi_{M+M^{-1}}(s^{-1}(s^2-t+1))  & \text{if $M^{*\omega}=M^{-1}$}.
\end{dcases}
\] 
\end{proposition}

\begin{proof}
This is verified from the fact that $\symplcharpoly^{\omega}_{M}(s,t)^2=\varphi_{\ssendo}(t)$.
\end{proof}
%
%
%
%
%
%
%
\section{The characteristic polynomial associated with two endomorphisms}\label{characteristic polynomial of two endomorphisms}
We have seen in Proposition~\ref{prop symplectic char poly} that  the square of the symplectic characteristic polynomial $\symplcharpoly^{\omega}_M(s,t)$
of an endomorphism $M$ 
is the characteristic polynomial $\varphi_{\ssendo}(t)$.
This section is devoted to the study of  the characteristic polynomial
 $\varphi_{\sendo{M}{N}}(t)$ 
for any endomorphisms $M,N$ of a (not necessarily symplectic) vector space.
\par
Let $V$ be a finite dimensional vector space over a field $\cfield$ with $\dim_{\cfield} V=n$. 
\begin{proposition}\label{prop relation gchar  char}
Let $M,N \in \operatorname{End}(V)$. 
Then
\begin{itemize}
\item
$\scharpoly{M}{N}{0}=\varphi_M(s)\varphi_N(s)$,
\item 
$\scharpoly{M}{N}{(\sigma-s)(\tau-s)}=\varphi_{MN}(\sigma\tau)+\dots + \varphi_{M+N}(\sigma+\tau)(-s)^{n}$.
\end{itemize}
\end{proposition}
\begin{proof}
It is clear from the definition.
\end{proof}
\begin{proposition}\label{prop_sim_poly}
Let $M,N\in\operatorname{End}(V)$.
If $M$, $N$  are simultaneously triangularizable, then
\[
\scharpoly{M}{N}{t}=\prod_{i=1}^n\{(\lambda_i-s)(\mu_i-s)-t\}, \quad \lambda_i, \mu_i \in \cfield. 
\]
In particular, a polynomial $(\lambda_i-s)(\mu_i-s)$ is 
an eigenvalue of $\sendo{M}{N}$.
\end{proposition}
\begin{proof}
Since $M$ and $N$ are simultaneously triangularizable, there exists a basis $\mathcal{B}$ such that 
\[
[M]_{\mathcal{B}}=
\left(
\begin{array}{ccccc}
\lambda_1& &  * \\
& \ddots \\
 0& & \lambda_n 
\end{array}
\right), \quad
[N]_{\mathcal{B}}=
\left(
\begin{array}{ccccc}
\mu_1& &  * \\
& \ddots \\
 0& & \mu_n 
\end{array}
\right).
\]
Therefore,  $\scharpoly{M}{N}{t}
=\prod_{i=1}^n\{(\lambda_i-s)(\mu_i-s)-t\}$.
\end{proof}

\begin{proposition}\label{prop all eigenvalue}
Let $M,N\in\operatorname{End}(V)$. Suppose that $MN=NM$.
The following are equivalent.
\begin{itemize}
\item [$(1)$]
All eigenvalues of $M$ and all eigenvalues of $N$ are in $\cfield$.
\item [$(2)$]
The characteristic polynomial $\scharpoly{M}{N}{t}$ is of the form 
\[
\scharpoly{M}{N}{t}=\prod_{i=1}^n\{(\lambda_i-s)(\mu_i-s)-t\},\quad \lambda_i,\mu_i \in \cfield.
\] 
\end{itemize}
\end{proposition}
\begin{proof}
It is clear that $(2)$ implies $(1)$.
If all eigenvalues of $M$, $N$ are in $\cfield$, then $M$, $N$ are simultaneously triangularizable. Hence, Proposition~\ref{prop_sim_poly} gives $(2)$.
\end{proof}

\begin{proposition}\label{sympl eigen}
Let $M,N \in \operatorname{End}(V)$ and $\lambda,\mu \in \cfield$.
If $(\lambda-s)(\mu-s)$ is 
an eigenvalue of $\sendo{M}{N}$,
then 
$\lambda\mu$ is an eigenvalue of $MN$ and $NM$ and
$\lambda+\mu$ is an eigenvalue of $M+N$.
%
Moreover, if 
the characteristic polynomial 
of $M$ 
is equal to 
that of $N$,
 then $\lambda, \mu$ are eigenvalues of $M$ and $N$.
\end{proposition} 
\begin{proof}
It is easy to see from Proposition~\ref{prop relation gchar  char}. 
%
\end{proof}
By using this proposition, we give an example of endomorphisms $M$, $N$ which satisfy that 
$
\scharpoly{M}{N}{(\lambda-s)(\mu-s)}\neq 0
$
for all $\lambda, \mu \in \cfield$.
\begin{example}
Let $V=\cfield^2$. Define two matrices $M$, $N$ by 
\[
M=
\left(
\begin{array}{cc}
0 & 1 \\
0 & 0  \\
\end{array}
\right), \
N=
\left(
\begin{array}{cccc}
0 & 0 \\
1 & 0 
\end{array}
\right).
\]
It is easily shown that
$
\varphi_M(t)=\varphi_N(t)=t^2$ and that 
$\varphi_{M+N}(t)=(1-t)(1+t).
$
Hence, 
Proposition~\ref{sympl eigen} shows
that 
$
\scharpoly{M}{N}{(\lambda-s)(\mu-s)}\neq 0
$
 for all $\lambda, \mu \in \cfield$.
We note that $\varphi_{\sendo{M}{N}}(t)=s^4-(1+2s^2)t+t^2$.
\end{example}

It is well-known that 
$(\lambda-s)(\mu-s)$ is an eigenvalue of $\sendo{M}{N}$
if and only if $\geigensp{\sendo{M}{N}}{(\lambda-s)(\mu-s)}\neq \{0\}$. We show that  if $MN=NM$, then 
\begin{itemize}
\item 
the space $\geigensp{\sendo{M}{N}}{(\lambda-s)(\mu-s)}$ is generated by vectors in $\geigensp{M,N}{\lambda, \mu}$,
\item
the space  $\geigensp{M,N}{\lambda, \mu}$ is equal to $\geigensp{MN}{\lambda\mu}\cap \geigensp{M+N}{\lambda+\mu}$
\end{itemize}
We first give a lemma.
\begin{lemma}\label{lemma subset eigensp}
Let $M,N\in \operatorname{End}(V)$ and $\lambda, \mu, \in\cfield$. Suppose that $MN=NM$. Then
\begin{align*}
\geigensp{M,N}{\lambda, \mu}&\subset\geigensp{MN}{\lambda\mu}\cap \geigensp{M+N}{\lambda+\mu}, \\
(\geigensp{MN}{\lambda\mu}\cap \geigensp{M+N}{\lambda+\mu})\otimes_{\cfield}\ccfield&\subset \geigensp{(M-s\idmap{})(N-s\idmap{})}{(\lambda-s)(\mu-s)}.
\end{align*}
\end{lemma}
\begin{proof}
To show the first implication,
it suffices to see that 
$
\geigensp{M}{\lambda} \cap \geigensp{N}{\mu}
\subset \geigensp{MN}{\lambda\mu}\cap \geigensp{M+N}{\lambda+\mu}$.
Let $W=\geigensp{M}{\lambda} \cap \geigensp{N}{\mu}$ and $M_1=M|_W$, $N_1=N|_W$. Since $M_1$, $N_1$ are simultaneously triangularizable, there exists a basis $\mathcal{B}_1$ of $W$ such that 
\[
[M_1]_{\mathcal{B}_1} 
=\left(
\begin{array}{cccc}
\lambda & & *\\
& \ddots  \\
0 &  & \lambda
\end{array}
\right), \ 
[N_1]_{\mathcal{B}_1} 
=\left(
\begin{array}{cccc}
\mu & & *\\
& \ddots  \\
0 &  & \mu
\end{array}
\right).
\] 
Hence, $W=\geigensp{M_1 N_1}{\lambda\mu}\cap \geigensp{M_1+N_1}{\lambda+\mu}\subset\geigensp{MN}{\lambda\mu}\cap \geigensp{M+N}{\lambda+\mu}$.
\par
Let $Z=(\geigensp{MN}{\lambda\mu}\cap \geigensp{M+N}{\lambda+\mu})\otimes_{\cfield}\ccfield$. 
It is easy to see that 
\[
\geigensp{\sendo{M}{N}}{(\lambda-s)(\mu-s)}=
\geigensp{MN-s(M+N)}{\lambda\mu-s(\lambda+\mu)} .
\]
 Since $MN|_Z$, $(M+N)|_Z$ are simultaneously triangularizable, in the same way as above,  
we get
$Z
\subset \geigensp{MN-s(M+N)}{\lambda\mu-s(\lambda+\mu)}$.
\end{proof}
\begin{theorem}\label{thm geigensp}
Under the  same assumption as the lemma above,
we have
\begin{align*}
\geigensp{M,N}{\lambda, \mu}&=\geigensp{MN}{\lambda\mu}\cap \geigensp{M+N}{\lambda+\mu}, \\
\geigensp{M,N}{\lambda, \mu}\otimes_{\cfield}{\ccfield}
&=\geigensp{\sendo{M}{N}}{(\lambda-s)(\mu-s)}.
\end{align*}
\end{theorem}
\begin{proof}
By Lemma~\ref{lemma subset eigensp}, it suffices to prove that 
\[
\geigensp{\sendo{M}{N}}{(\lambda-s)(\mu-s)}\subset  \geigensp{M,N}{\lambda, \mu}\otimes_{\cfield}\ccfield.
\]
Let $W=\geigensp{\sendo{M}{N}}{(\lambda-s)(\mu-s)}$ and let $M_1=M|_W$, $N_1=N|_W$. 
Then 
\begin{equation}\label{eq char poly}
\varphi_{M_1}(s)\varphi_{N_1}(s)=(\lambda-s)^k(\mu-s)^k.
\end{equation}
In the case where $\lambda=\mu$,  this implies that
$
W=\geigensp{M_1}{\lambda}\cap \geigensp{N_1}{\lambda}=\geigensp{M_1,N_1}{\lambda,\lambda }\subset\geigensp{M,N}{\lambda, \lambda}\otimes_{\cfield}\ccfield.
$
Suppose that $\lambda\neq \mu$. Then  the equality \eqref{eq char poly} implies that
$W$ is the sum of $\geigensp{M_1,N_1}{\lambda,\mu}$, $\geigensp{M_1}{\lambda}\cap\geigensp{N_1}{\lambda}$, $\geigensp{M_1}{\mu}\cap\geigensp{N_1}{\mu}$.
Applying Lemma~\ref{lemma subset eigensp} for $M_1$, $N_1$, we get
\[
\geigensp{M_1}{\lambda}\cap\geigensp{N_1}{\lambda} \subset \geigensp{(M_1-s\idmap{})(N_1-s\idmap{})}{\lambda-s)(\lambda-s)}=\{0\}.
\] 
In the same way, $\geigensp{M_1}{\mu}\cap\geigensp{N_1}{\mu} =\{0\}$.
Hence, we get $W=\geigensp{M_1,N_1}{\lambda, \mu}\subset\geigensp{M,N}{\lambda, \mu}\otimes_{\cfield}\ccfield$.
\end{proof}

As a consequence of Theorem~\ref{thm geigensp}, we have the following.
\begin{corollary}\label{corollary decomposition}
Let $M,N\in \operatorname{End}(V)$. Suppose that $MN=NM$ and that
\[
\scharpoly{M}{N}{t}=\prod_{i=1}^k\{(\lambda_i-s)(\mu_i-s)-t\}^{m_i}
\]  
so that  $(\lambda_i-s)(\mu_i-s) \neq (\lambda_j-s)(\mu_j-s)$ for $i\neq j$. Then we have the following.
\begin{itemize}
\item 
The space $V$ is the direct sum of the $m_i$-dimensional $(M,N)$-invariant subspaces $\geigensp{M,N}{\lambda_i, \mu_i}$.
\item
Let $P_1,\dots,P_k$ be  projections with respect to $V=\bigoplus_{i=1}^k \geigensp{M,N}{\lambda_i,\mu_i}$ and let $Q_1,\dots, Q_k$ be projections with respect to 
\[
V\otimes_{\cfield}\ccfield=\bigoplus_{i=1}^k\geigensp{\sendo{M}{N}}{(\lambda_i-s)(\mu_i-s)}.
\]
Then 
$
P_i=Q_i.
$
\end{itemize}
\end{corollary}

\begin{proposition}\label{prop eigenvalue eigenvector}
Let $M,N\in \operatorname{End}(V)$ and $\lambda, \mu \in \cfield$. Suppose that $MN=NM$. Then
$(\lambda-s)(\mu-s)$ is an eigenvalue of $\sendo{M}{N}$
if and only if $\eigensp{M,N}{\lambda , \mu}\neq\{0\}$.
\end{proposition}
\begin{proof}
Suppose that $\eigensp{M,N}{\lambda , \mu}\neq\{0\}$.  Let $v\in (\eigensp{M,N}{\lambda , \mu})\otimes_{\cfield}\ccfield$. Then we get $\sendo{M}{N}v=(\lambda-s)(\mu-s)v$. Hence, $(\lambda-s)(\mu-s)$ is an eigenvalue of $\sendo{M}{N}$. 
\par
The only if part follows from 
Theorem~\ref{thm geigensp}
 and the fact that $\eigensp{M,N}{\lambda,\mu}=\{0\}$ is equivalent to $\geigensp{M,N}{\lambda,\mu}=\{0\}$.
\end{proof}


\section{Applications of the symplectic characteristic polynomial}\label{Applications of the symplectic characteristic polynomial}
In this section, we prove several results 
for symplectically normal endomorphisms 
which are  concerned with the symplectic characteristic polynomial.
\par
Let $(V,\omega)$ be a $2n$-dimensional symplectic vector space 
over a field $\cfield$. 
\begin{proposition}
Let $M\in\operatorname{End}(V)$. Suppose that  $MM^{*\omega}=M^{*\omega}M$. 
The following are equivalent.
\begin{itemize}
\item [$(1)$]
All eigenvalues of $M$ are in $\cfield$.
\item [$(2)$]
The characteristic polynomial $\symplcharpoly^{\omega}_M(s,t)$ is of the form 
\[
\symplcharpoly^{\omega}_M(s,t)=\prod_{i=1}^n\{(\lambda_i-s)(\mu_i-s)-t\}, \quad \lambda_i,\mu_i \in \cfield.
\] 
\end{itemize}
\end{proposition}
\begin{proof}
This follows from $\symplcharpoly^{\omega}_M(s,t)^2=\varphi_{\ssendo}(t)$ and Proposition~\ref{prop all eigenvalue}.
\end{proof}

Let $M \in \operatorname{End}(V)$. Suppose that all eigenvalues of $M$ are in $\cfield$.
By using the same idea of the proof of Lemma~1  in \cite{gohberg1990classification}, we can show that $V$ is the symplectically orthogonal direct sum of the symplectic subspaces $\symplgeigensp{\omega}{M}{\lambda}{\mu}$ associated with  two eigenvalues $\lambda, \mu$ of $M$.
We reformulate this result with the symplectic characteristic polynomial and give a short proof.
\begin{lemma}\label{lemma sympl decomp}
Let $M\in\operatorname{End}(V)$. 
Suppose that $MM^{*\omega}=M^{*\omega}M$. If  the symplectic characteristic polynomial $\symplcharpoly^{\omega}_M(s,t)$ is of the form
\[
\symplcharpoly^{\omega}_M(s,t)=\prod_{i=1}^k\{(\lambda_i-s)(\mu_i-s)-t\}^{m_i}, \quad \lambda_i,\mu_i \in \cfield
\]  
so that  $(\lambda_i-s)(\mu_i-s) \neq (\lambda_j-s)(\mu_j-s)$ for $i\neq j$, then 
$V$ is the symplectically orthogonal direct sum of  the $2m_i$-dimensional symplectic subspaces $\symplgeigensp{\omega}{M}{\lambda_i}{\mu_i}$.  
Precisely, 
we have the following.
\begin{itemize}
\item [$(1)$]
$\dim_{\cfield}\symplgeigensp{\omega}{M}{\lambda_i}{\mu_i}=2m_i$.
\item[$(2)$] 
$V=\bigoplus_{i=1}^k \symplgeigensp{\omega}{M}{\lambda_i}{\mu_i}$.
\item[$(3)$] 
$\symplgeigensp{\omega}{M}{\lambda_i}{\mu_i} \subset\left(\symplgeigensp{\omega}{M}{\lambda_j}{\mu_j}\right)^{\perp\omega} $ \quad 
for $i \neq j$.
\item[$(4)$] 
$\symplgeigensp{\omega}{M}{\lambda_i}{\mu_i}$ is an $(M,M^{*\omega})$-invariant symplectic subspace of $(V,\omega)$.
\end{itemize} 
\end{lemma}

\begin{proof}
Let $M_s^{\omega}=(M-s\idmap{})(M^{*\omega}-s\idmap{})$.
By Theorem~\ref{thm geigensp},
 we have
\begin{equation}\label{eq generalize eigensp}
\symplgeigensp{\omega}{M}{\lambda_i}{\mu_i}\otimes_{\cfield}\ccfield
=\geigensp{M_s^{\omega}}{(\lambda_i-s)(\mu_i-s)}.
\end{equation}
Since $\symplcharpoly^{\omega}_M(s,t)^2=\varphi_{\ssendo}(t)$, this implies $(1)$ and $(2)$.
The item $(5)$ in 
Proposition~\ref{propperp} shows that 
$
\geigensp{M_s^{\omega}}{(\lambda_i-s)(\mu_i-s)}
\subset 
\left(\geigensp{M_s^{\omega}}{(\lambda_j-s)(\mu_j-s)}\right)^{\perp\omega}.
$
Moreover,
Proposition~\ref{prop skew ham sympl subsp} gives that $\geigensp{M_s^{\omega}}{(\lambda_i-s)(\mu_i-s)}
$ is a symplectic subspace of $(\exspace, \omega)$. 
Hence, from \eqref{eq generalize eigensp}, we get $(3)$ and $(4)$.
\end{proof}
In the symplectic case, we have a stronger result than Proposition~\ref{prop eigenvalue eigenvector}.
\begin{proposition}\label{lemma sympl eigen vector}
Let $M\in \operatorname{End}(V)$ and $\lambda,\mu\in \cfield$.
Suppose that $MM^{*\omega}=M^{*\omega}M$. Then
$
\symplcharpoly^{\omega}_M(s,(\lambda-s)(\mu-s))=0$ if and only if $\eigensp{M}{\lambda}\cap \eigensp{M^{*\omega}}{\mu} \neq \{0\}$.
\end{proposition}
\begin{proof}
By Proposition~\ref{prop eigenvalue eigenvector}, it is enough to show that $\sympleigensp{\omega}{M}{\lambda}{\mu}\neq \{0\}$ implies $\eigensp{M}{\lambda}\cap \eigensp{M^{*\omega}}{\mu} \neq \{0\}$. The case where $\lambda=\mu$ is clear. Suppose that $\lambda \neq \mu$.
Lemma~\ref{lemma sympl decomp} shows that 
$\symplgeigensp{\omega}{M}{\lambda}{\mu}$ is a symplectic subspace of $(V,\omega)$. 
Hence, by Lemma~\ref{lemma Lag split}, we have 
$\geigensp{M}{\lambda}\cap \geigensp{M^{*\omega}}{\mu} \neq \{0\}$. Since $MM^{*\omega}=M^{*\omega}M$, this implies
$\eigensp{M}{\lambda}\cap \eigensp{M^{*\omega}}{\mu} \neq \{0\}$.
\end{proof}
It is proved in \cite[Theorem~13]{de2016diagonalizability} that if an endomorphism is symplectically normal and diagonalizable, then the endomorphism is symplectically diagonalizable. 
The following theorem is an improvement of this fact.
\begin{theorem}\label{theorem complete inv}
Let $M\in \operatorname{End}(V)$.  Suppose that $MM^{*\omega}=M^{*\omega}M$ and that $M$ is diagonalizable. 
Then there exists a symplectic basis  
$(e_1,\dots, e_n, f_1,\dots,f_n)$
such that 
\[
Me_i=\lambda_i e_i, \ M f_i=\mu_i f_i, \quad i\in \{1\dots, n\}
\]
where the symplectic characteristic polynomial  $\symplcharpoly^{\omega}_M(s,t)$ is of the form
\[
\symplcharpoly^{\omega}_M(s,t)=\prod_{i=1}^n\{(\lambda_i-s)(\mu_i-s)-t\}, \quad \lambda_i,\mu_i \in \cfield.
\]
\end{theorem}
\begin{proof}
We proof by induction on $n$.
The item $(4)$ in
Lemma~\ref{lemma sympl decomp} shows that  $\sympleigensp{\omega}{M}{\lambda_1}{\mu_1}$ is a symplectic subspaces of $(V,\omega)$.
There exist $e_1,f_1\in V$ such that 
\[
e_1\in \eigensp{M}{\lambda_1}\cap \eigensp{M^{*\omega}}{\mu_1}, \ f_1\in \eigensp{M}{\mu_1}\cap \eigensp{M^{*\omega}}{\lambda_1}, \ \omega(e_1,f_1)=1.
\]
Indeed, the case where $\lambda=\mu$ follows from Proposition~\ref{prop exist sympl basis}  and the case where $\lambda \neq \mu$ follows from Lemma~\ref{lemma Lag split} and \ref{lemma sympl basis}.
Hence, the assertion holds for $n=1$.
We assume that the assertion holds for $n-1$.
Let $W$ be a subspace generated by $e_1,f_1$. Then $W$ is an $(M,M^{*\omega})$-invariant symplectic subspace of $(V,\omega)$ and by the item $(4)$ in Proposition~\ref{propperp}, so is $W^{\perp\omega}$. 
Let $M_1=M|_{W^{\perp\omega}}$ and $\omega_1=\omega|_{W^{\perp\omega}}$.
Since $M_1^{*\omega_1}=M^{*\omega}|_{W^{\perp\omega}}$, we have
$M_1M_1^{*\omega_1}=M_1^{*\omega_1}M_1$.
It is clear that  $M_1$ is diagonalizable.  The induction hypothesis shows that there exists a symplectic basis 
$(e_2,\dots,e_n,f_1,\dots,f_n)$ of $(W^{\perp\omega}, \omega|_{W^{\perp\omega}})$ such that 
$
Me_i=\lambda_i e_i, \ M f_i=\mu_i f_i
$
for $i\in \{2,\dots, n\}$.
It is clear that $(e_1,\dots,e_n,f_1,\dots, f_n)$ is a symplectic basis which is desired.
\end{proof}
\begin{corollary}\label{corollary sympl diago}
The symplectic characteristic polynomial is a complete invariant with respect to symplectic similarity for symplectically diagonalizable endomorphisms. 
In particular, a symplectically diagonalizable endomorphism $M$ is symplectically similar to the symplectic adjoint endomorphism $M^{*\omega}$.
\end{corollary}
\begin{proof}
Let $M$, $N$ be symplectically diagonalizable endomorphisms. Suppose that $\symplcharpoly^{\omega}_M(s,t)=\symplcharpoly^{\omega}_N(s,t)$. Theorem~\ref{theorem complete inv} shows that there exist symplectic bases $\mathcal{B}$, $\mathcal{B}'$ such that $[M]_{\mathcal{B}}=[N]_{\mathcal{B}'}$. Hence, Proposition~\ref{prop sympl similar} gives that $M$ is symplectically similar to $N$. 
\end{proof}

It is well-known that if the number of distinct eigenvalues of an endomorphism is equal to $\dim V$, then the endomorphism is diagonalizable. 
In the symplectic case, we have the following two theorems.
\begin{theorem}\label{theorem simul sympl diagonal}
Let $M\in \operatorname{End}(V)$.
Suppose that $MM^{*\omega}=M^{*\omega}M$ and that the symplectic characteristic polynomial $\symplcharpoly^{\omega}_M(s,t)$ is of the form
\[
\symplcharpoly^{\omega}_M(s,t)=\prod_{i=1}^n\{(\lambda_i-s)(\mu_i-s)-t\}, \quad \lambda_i,\mu_i \in \cfield
\] 
and 
$
(\lambda_i-s)(\mu_i-s) \neq (\lambda_j-s)(\mu_j-s) $ for $i\neq j.
$
Then 
there exists a symplectic basis 
$(e_1,\dots, e_n,f_1,\dots, f_n)$ 
such that 
\begin{align*}
MM^{*\omega}e_i&=\lambda_i\mu_i e_i, & (M+M^{*\omega})e_i&=(\lambda_i+\mu_i) e_i, \\
MM^{*\omega}f_i&=\lambda_i\mu_i f_i, & (M+M^{*\omega})f_i&=(\lambda_i+\mu_i) f_i, & i\in\{1,\dots,n\}.
\end{align*}
\end{theorem}

\begin{proof}
Let $W_i=\symplgeigensp{\omega}{M}{\lambda_i}{\mu_i}$ and let $M_i=M|_{W_i}$, $\omega_i=\omega|_{W_i}$.
Lemma~\ref{lemma sympl decomp} shows that $\dim_{\cfield} W_i=2$ and that $W_i$ is symplectic. Hence, we get $\symplcharpoly^{\omega_i}_{M_i}(s,t)=\{(\lambda_i-s)(\mu_i-s)-t\}$, which implies that  $\varphi_{M_i}(t)=(\lambda_i-s)(\mu_i-s)$.  
It follows from $M^{*\omega}|_{W_i}=(M_i)^{*\omega_i}$ that 
\begin{align*}
(MM^{*\omega})|_{W_i}&=M_i M_i^{*\omega_i}=(\det M_i) \idmap{}=\lambda_i\mu_i \idmap{}, \\
(M+M^{*\omega})|_{W_i}&=M_i+M_i^{*\omega_i}=(\trace  M_i) \idmap{}=(\lambda_i+\mu_i) \idmap{}.
\end{align*}
Since $W_i$ is symplectic, there exist $e_i,f_i\in W_i$ such that $\omega(e_i,f_i)=1$. The item $(3)$ in Lemma~\ref{lemma sympl decomp} shows that  
$(e_1,\dots, e_n,f_1\dots, f_n)$ is a symplectic basis, which completes the proof.
\end{proof}
\begin{theorem}\label{theorem sympl diagonal}
Under the same assumption as above, we assume further that
$\lambda_i \neq \mu_i$ for all $i$.
Then there exists a symplectic basis $(e_1,\dots, e_n, f_1,\dots, f_n)$ such that
\begin{align*}
Me_i&=\lambda_i e_i, & M^{*\omega}e_i&=\mu_i e_i, & \\
 M f_i&=\mu_i f_i, & M^{*\omega}f_i&=\lambda_i f_i, & i\in\{1,\dots,n\}.
\end{align*}
\end{theorem}
\begin{proof}
Let $W_i=\symplgeigensp{\omega}{M}{\lambda_i}{\mu_i}$. 
Lemma~\ref{lemma sympl decomp} gives that $\dim_{\cfield}W=2$ and that $W_i$ is a symplectic subspace.  Hence, Proposition~\ref{lemma sympl eigen vector} shows that
%
\[
\dim_{\cfield}\left(\eigensp{M}{\lambda_i}\cap \eigensp{M^{*\omega}}{\mu_i}\right)=\dim_{\cfield}\left(\eigensp{M}{\mu_i}\cap \eigensp{M^{*\omega}}{\lambda_i}\right)=1.
\]
Since $W_i$ is symplectic, 
there exist $e_i, f_i\in V$ such that 
\[
e_i \in \eigensp{M}{\lambda_i}\cap \eigensp{M^{*\omega}}{\mu_i}, \ 
f_i \in  \eigensp{M}{\mu_i}\cap \eigensp{M^{*\omega}}{\lambda_i}, \ \omega(e_i, f_i)=1.
\]
The item $(3)$ in Lemma~\ref{lemma sympl decomp} shows that $(e_1,\dots, e_n, f_1,\dots, f_n)$ is a symplectic basis which is desired.
\end{proof}

\bibliographystyle{plain} 
\bibliography{SCP2024ref}

\end{document}